\documentclass[a4paper]{amsart}
\usepackage{amscd,amssymb,amsmath,mathrsfs,latexsym,url}
\usepackage[all]{xy}
\usepackage{color,graphicx}
\usepackage{hyperref}
\def\ov#1{{\overline{#1}}}

\newcommand{\rec}{\operatorname{rec}}

\newcommand{\cco}{\operatorname{c^{\circ}}}
\newcommand{\cc}{\operatorname{c}}

\newcommand{\aff}{\operatorname{aff}}

\newcommand{\Spec}{\operatorname{Spec}}

\renewcommand{\i}{\operatorname{i}}

\newcommand{\ri}{\operatorname{ri}}
\newcommand{\trop}{\operatorname{trop}}

\newcommand{\Hom}{\operatorname{Hom}}

\newcommand{\val}{{\operatorname{val}}}

\def \G{\mathbb{G}}

\def \Q{\mathbb{Q}}
\def \R{\mathbb{R}}

\def \T{\mathbb{T}}
\def \Z{\mathbb{Z}}

\def\cV {{\mathcal V}}
\def\cX {{\mathcal X}}

\numberwithin{equation}{section}
\theoremstyle{definition}
\newtheorem{defn}[equation]{Definition}

\newtheorem{rem}[equation]{Remark}
\newtheorem{exmpl}[equation]{Example}

\theoremstyle{plain}
\newtheorem{lem}[equation]{Lemma}

\newtheorem{prop}[equation]{Proposition}
\newtheorem{thm}[equation]{Theorem}
\newtheorem{cor}[equation]{Corollary}
\newtheorem{prop-def}[equation]{Proposition-Definition}

\begin{document}

\title[Complexes and fans]{When do the recession cones of a polyhedral complex form a fan?}

\author[Jos\'e Ignacio Burgos Gil]{Jos\'e Ignacio Burgos Gil}
\address{Instituto de Ciencias Matem\'aticas (CSIC-UAM-UCM-UC3M).
 Calle Nicol\'as Cabrera 13-15,
 Cantoblanco,
 28049 Madrid, Spain}
\email{burgos@icmat.es}
\urladdr{}
\author[Mart{\'\i}n~Sombra]{Mart{\'\i}n~Sombra}
\address{ICREA \& Universitat de Barcelona, Departament d'{\`A}lgebra
  i Geometria. 
Gran Via 585, 08007 Bar\-ce\-lo\-na, Spain}
\email{sombra@ub.edu}
\urladdr{\url{http://atlas.mat.ub.es/personals/sombra}}

\thanks{Burgos Gil and Sombra were partially supported by the MICINN
  research project MTM2009-14163-C02-01. Burgos Gil was also partially
  supported by the CSIC research project 2009501001.}


\begin{abstract}
We study the problem of when the collection of the recession cones of a
polyhedral complex forms also a complex. We exhibit an example showing
that this is no always the case. We also show that if the support of
the given polyhedral complex satisfies a Minkowski-Weyl type
condition, then the answer is positive. As a consequence, we obtain a
classification theorem for proper toric schemes over a discrete
valuation ring in terms of complete strongly convex rational
polyhedral complexes.  
\end{abstract}
\maketitle

\overfullrule=0.3mm
\vspace{-8mm}

\section{Introduction}
\label{sec:introduction}

Toric schemes over a discrete valuation ring (DVR) were introduced and
classified in \cite{Kempfals:te}.  Let $N\simeq \Z^{n}$ be a lattice
on rank $n$ and set $N_{\R}=N\otimes \R $ for the associated real vector space.
In \emph{loc. cit.}, toric schemes over a DVR are described and
classified in 
terms of rational fans in $N_{\R}\times \R_{\ge 0}$. The toric scheme
associated to a fan
is proper if and only the support of the fan is the whole
of $N_{\R}\times \R_{\ge 0}$. In this case, we say that the fan is
\emph{complete}. 
In the literature, such toric schemes are also called toric
degenerations.  

If we intersect a rational fan in $N_{\R}\times \R_{\ge 0}$ with the hyperplane
$N_{\R}\times \{1\}$ we obtain a strongly convex rational (SCR) polyhedral
complex in $N_{\R}$. Thus, given a SCR polyhedral
complex in $N_{\R}$, it is natural to ask if it defines a rational
fan in  $N_{\R}\times \R_{\ge 0}$ and hence a toric scheme. One can
also ask up to which extent these complexes classify toric schemes
over a DVR. 

In \cite[Lemma 3.2]{NishinouSiebert:tdtvtc}, T. Nishinou and
B. Siebert claim that any complete 
SCR polyhedral complex in $N_{\R}$ gives rise
to a rational fan in  $N_{\R}\times \R_{\ge 0}$, but their proof is
incomplete because they do not check that the intersection of any two
of the obtained cones is a common face. 
Moreover, in recent references, 
it is claimed without proof that any SCR  polyhedral complex (without
further hypothesis) defines a 
rational fan, see for instance \cite[\S~2.3]{SpeyerPhD}.  

This question is less innocent than it appears. Indeed, in Example
\ref{exm:17} we exhibit a SCR polyhedral complex that does not define a
rational fan. Hence it is not possible to associate a toric scheme to an
arbitrary SCR polyhedral complex. 

The main result of this note (Theorem \ref{thm:14}) is that, if we
assume that the support 
of a  complex is connected and satisfies the Minkowski-Weyl condition (see
Definition \ref{def:41} below) then its recession cones do form a
complex. In particular, 
this is true in the complete case, thus filling the gap in the
proof of the Nishinou-Siebert statement.
As a consequence, we show that proper toric schemes
over a DVR of relative dimension~$n$ are classified by complete SCR
polyhedral complexes in $N_{\R}$ (Theorem \ref{thm:7}).

The connectedness and the Minkowski-Weyl hypothesis are sufficient but
not necessary. A wide 
class of examples whose recession cones form a complex but that may
not satisfy the above hypothesis is that of extendable
complexes (Definition \ref{def:3}). A consequence of our result is
that any extendable SCR polyhedral complex
defines a toric scheme. 

The question of when an SCR polyhedral complex defines a fan is also
relevant in tropical geometry: the definition of a tropical
compactification of a subvariety of the torus depends on the
construction of a toric scheme from 
an SCR polyhedral complex supported on the associated tropical
variety. Since a tropical variety has a natural structure of SCR
polyhedral complex that is 
extendable, our theorem implies the existence
 of such a toric scheme. 


\subsection*{Acknowledgements} We thank Eric Katz and Francisco Santos 
for enlightening discussions. We also thank an anonymous referee for
useful suggestions. This research was done during a stay of the first
author at the University of Barcelona. We thank this institution for
its hospitality. 

\section{Preliminaries on polyhedral complexes} 
\label{sec:defin-prel} 
Let $N\simeq \Z^{n}$ be a lattice on rank $n$.
We write $N_{\R}=N\otimes \R $ for the associated real vector space
and let  
$M_{\R}=N_{\R}^{\vee}$ be its dual space. The pairing between $x\in
M_{\R}$ and $u\in N_{\R}$ 
 will be denoted by  
$\langle x,u\rangle$.
A \emph{polyhedron} of $N_{\R}$ is
  a convex set defined as the intersection of a finite number of
  closed halfspaces.  It is  \emph{strongly convex} if it 
contains no line. It is \emph{rational} if the closed halfspaces can be
chosen to be defined by equations with coefficients in $\Q$. 
 A \emph{polyhedral set} is a finite
union of polyhedra.
A \emph{polytope} is the convex hull of a finite set of points. A
\emph{convex polyhedral cone} is the convex conic set 
generated by a finite set of vectors of $N_{\R}$.

The following theorem is a basic tool in the study of polyhedra.

\begin{thm}[Minkowski-Weyl] \label{thm:1}
Let $E$ be a subset of $N_{\R}$.
Then $E$ is a  polyhedron if and only if there exists a polytope
$\Delta$ and a convex polyhedral cone 
$\sigma$ such that 
\begin{displaymath}
E= \Delta+\sigma.
\end{displaymath}
\end{thm}

An immediate consequence of this result is that  
the notion of polytope coincides with that of bounded
 polyhedron and the notion of 
convex polyhedral cone coincides with that of conic 
polyhedron. 

Let $\Lambda$ be a  polyhedron. The \emph{relative interior} of
$\Lambda$, denoted $\ri(\Lambda)$, is defined as the interior of
$\Lambda$ relative to the minimal affine space that contains it. 
For $x\in M_{\R}$, we set 
\begin{displaymath}
\Lambda_{x}=\{u\in
\Lambda\mid \langle x,u\rangle\le \langle x,v\rangle,\, \forall v\in 
\Lambda\}.  
\end{displaymath}
A 
non-empty subset $F\subset
\Lambda$ is called a \emph{face} of $\Lambda$ if it is of the form
$\Lambda_{x}$ for some $x\in M_{\R}$.
The polyhedron $\Lambda$ is the disjoint union of the relative
interior of its faces.

Let $E\subset N_{\R}$ be a polyhedral subset. For each $p\in E$, the
\emph{local recession cone} of $E$ at $p$ is 
defined as
  \begin{displaymath}
    \rec_{p}(E)=\{u\in N_{\R}\mid p+\lambda u \in E,\,\forall \lambda
    \ge 0\}.
  \end{displaymath}
The \emph{recession
  cone} of $E$ is defined as 
\begin{displaymath}
\rec(E)=\bigcap_{p\in E} \rec_{p}(E).  
\end{displaymath}
Observe that both the local and the global recession cones are conic
subsets of $N_{\R}$. 
If $\Lambda $ is a  polyhedron, we have the alternative
characterization 
\begin{displaymath}
  \rec(\Lambda )=\{u\in N_{\R}\mid \Lambda +u\subset \Lambda \}.
\end{displaymath}
Hence, for  polyhedra, the above definition agrees with the 
notion of recession cone from convex analysis \cite{Roc70}.
Moreover, 
$\rec(\Lambda)$ agrees with the cone $\sigma $ in Theorem~\ref{thm:14}. In 
particular, $\sigma$ is determined by $\Lambda$ and $\rec(\Lambda )$
is a convex polyhedral cone.   
A further consequence of Minkowski-Weyl Theorem is that, for 
polyhedra, local and global 
recession cones agree. In other words,  
$\rec_{p}(\Lambda)=\rec(\Lambda)$ for any $p\in \Lambda$.
Observe  that, if $\Lambda $ is rational or strongly
convex, the same is true for $\rec(\Lambda)$. 

To a polyhedral subset $E\subset N_{\R}$ we can also associate a
(non-necessarily convex) polyhedral cone 
of $N_{\R}\times \R$, given by
\begin{displaymath}
  \cc(E) =\ov{\R_{>0}(E\times \{1\})}\subset N_{\R}\times
\R_{\ge 0}.
\end{displaymath}
It is called the \emph{cone} of $E$.
If $\Lambda $ is a  polyhedron, then $\rec(\Lambda )=\cc(\Lambda
)\cap (N_{\R}\times \{0\})$. This is not true for general polyhedral
sets. 
Again, if $\Lambda $ is rational or strongly
convex,  the same is true for $\cc(\Lambda)$.

\begin{defn}
  \label{def:1}
  A \emph{polyhedral complex} in $N_{\R}$
  is a non-empty collection~$\Pi $ of 
  polyhedra of $N_{\R}$ such that
\begin{enumerate}
\item \label{item:58} every face of an element of $\Pi $
is also in $\Pi$,
\item \label{item:59} any  two elements of $\Pi$ are either
disjoint or intersect in a common face.
\end{enumerate}   
A polyhedral complex $\Pi$ is called \emph{rational} 
(respectively \emph{strongly convex}, \emph{conic}) if all of
its elements are rational (respectively strongly convex, cones). A
strongly 
convex conic polyhedral complex is called a \emph{fan}. For shorthand,
strongly convex rational will be abbreviated to SCR.
\end{defn}

  The \emph{support} of $\Pi$ is
the polyhedral set 
\begin{displaymath}
|\Pi|=\bigcup_{\Lambda \in \Pi } \Lambda.
\end{displaymath}
For a subset $E\subset N_{\R}$,
we 
say that $\Pi $ is a polyhedral complex \emph{in} 
 $E$ whenever 
$|\Pi|\subset E$. 
We say that a complex in
 $E$ is \emph{complete} if $|\Pi|= E$.

\section{Complexes and fans}
\label{sec:complexes-fans}

There are two natural processes for linearizing a polyhedral
complex. Intuitively, the first one is to look at the
complex from far away, so that the 
unbounded polyhedra became cones. In precise terms,  
the \emph{recession} of $\Pi$ is defined as the collection of
polyhedral cones of $N_{\R}$
\begin{displaymath}
  \rec(\Pi )=\{\rec(\Lambda )\mid \Lambda\in \Pi \}. 
\end{displaymath}

The second process is analogous to the linearization of an affine
space. The \emph{cone}
of $\Pi$ is defined as the collection of cones 
in  $N_{\R}\times\R$ 
\begin{displaymath}
 \cc (\Pi)= \{\cc(\Lambda)\mid \Lambda \in
\Pi\}\cup \{\sigma\times\{0\}\mid \sigma\in \rec(\Pi)\}.
\end{displaymath}

It is a natural question to ask  whether the recession or the cone of
a given polyhedral complex is a complex too. The 
following example shows that this is not always the case.

\begin{exmpl}
  \label{exm:17}
  Let $\Pi$ be the polyhedral complex in $\R^{3}$
  consisting in the set of faces of the polyhedra 
  \begin{displaymath}
    \Lambda_{1}:=\{(x_{1},x_{2},0) |
    \,  x_{1},x_{2}\ge0\} , \quad 
    \Lambda_{2}:=\{( x_{1},x_{2},1) | \, x_{1}+x_{2},
    x_{1}-x_{2}\ge 0\}   . 
  \end{displaymath}
  We have that $\rec(\Lambda_{1})$ and $\rec(\Lambda_{2})$
  are two  cones in $\R^{2}\times \{0\}$
  whose intersection is the cone 
  $\{(x_{1},x_{2},0) |
  x_{2},x_{1}-x_{2}\ge0\} $. This cone is neither a face of  $\rec(\Lambda_{1})$
  nor of $\rec(\Lambda_{2})$. Hence $\rec(\Pi)$ is not a complex and
  consequently, neither is $\cc(\Pi)$.
  In figure~\ref{fig:example} we see the polyhedron $\Lambda _{1}$
  in light grey, the polyhedron $\Lambda _{2}$ in darker grey and
  $\rec(\Lambda _{2})$ as dashed lines. 
\end{exmpl}

  \begin{figure}[!h]
    \centering
  \input{ejemplo.pspdftex}  
    \caption{\ }  
    \label{fig:example}
  \end{figure}

This example shows that we
need to impose some condition on $\Pi$
if we want to ensure  that $\rec(\Pi )$ and $\cc(\Pi)$ are
complexes.  
The precise hypothesis in
Theorem~\ref{thm:14} was suggested to us by Francisco Santos. The key
observation is that it is enough to assume that $|\Pi|$ satisfies a
version of the Minkowski-Weyl theorem.

\begin{lem}\label{lemm:18}
  Let $E$ be a polyhedral set. The following conditions
  are equivalent.
  \begin{enumerate}
  \item \label{item:1} There is a decomposition 
  $E=\Delta + \sigma$, where $\Delta $ is a
  finite union of polytopes and $\sigma$ is a  convex polyhedral cone.
  \item \label{item:2} $\rec_{p}(E)= \rec(E)$ for all $p\in E$. 
  \end{enumerate}
\end{lem}

\begin{proof} We first prove that \eqref{item:1} implies
  \eqref{item:2}.
  Let $E=\Delta +\sigma$ be as in \eqref{item:1}. Since $\sigma$ is a 
  convex cone, it is a semigroup. This implies
  that $\sigma\subset \rec_{p}(E)$ for all $p\in E$.
  Since  $\Delta  $ is compact and $\sigma$ is closed, if $v\in 
  \rec_{p}(E)$, 
  the fact that the ray 
  $p+\R_{\ge 0} v$ is contained in $\Delta +\sigma$ implies that $v\in
  \sigma$. Hence $\rec_{p}(E)= \sigma$ for all $p$ and so
  $\rec(E)=\sigma= \rec_{p}(E)$. 

Conversely, write $E=\bigcup_{i}\Lambda_{i}$. We set
$\sigma=\rec(E)= \rec_{p}(E)$ for any $p\in E$ and  
$\sigma_{i}= \rec(\Lambda_{i})$ for each $i$. We have
that $\sigma_{i}\subset \sigma $ for all $i$. 
 By Minkowski-Weyl Theorem, for each  $i$, there exists a polytope
 $\Delta_{i} 
 \subset \Lambda_{i}$ such that $\Lambda_{i}=
 \Delta_{i}+\sigma_{i}$. Consider the finite union of polytopes
$\Delta=\bigcup_{i} \Delta_{i}$.
It is clear that  $E\subset \Delta+\sigma$. 
Besides, $\Delta+\sigma\subset E$ since $\Delta$ is contained in $E$.
Hence $E=\Delta+\sigma$, as stated. 
\end{proof}

\begin{defn}\label{def:41} 
A polyhedral subset $E$ of $ N_{\R}$ satisfies the
  \emph{Minkowski-Weyl condition} if it verifies any of the
  equivalent conditions in Lemma~\ref{lemm:18}.
\end{defn}

\begin{thm}\label{thm:14}
Let $\Pi $ be a polyhedral complex in $N_{\R}$ such that $|\Pi |$ is
a connected polyhedral set satisfying the Minkowski-Weyl condition. Then  
\begin{enumerate}
\item \label{item:3} $\rec(\Pi )$ is a conic
  polyhedral complex in $N_{\R}$ and
  \begin{math}
    |\rec(\Pi)|=\rec(|\Pi|);
  \end{math}
\item  \label{item:4} $ \cc( \Pi)$ is a conic polyhedral
  complex in $N_{\R}\times \R_{\ge 0}$  and
  \begin{math}
 |\cc(\Pi)|=\cc(|\Pi|).   
  \end{math}
\item \label{item:5} If, in addition, $\Pi$ is rational (respectively
  strongly convex), then both $\rec(\Pi )$ and $\cc(\Pi)$ are rational
  (respectively fans).
\end{enumerate}
\end{thm}

We need some lemmas before starting the proof of this result.
\begin{lem}
  \label{lemm:12}
Let $\Lambda\subset N_{\R}$ be a  polyhedron.
Then the 
collection of the recession cones of the form
$\rec(F )$ for some face $F$ of $\Lambda$ coincides with the set of
faces of $\rec(\Lambda)$. In particular, it is a complex.
\end{lem}

\begin{proof}
By Minkowski-Weyl Theorem, there exist a polytope $\Delta$ and a
cone $\sigma$ such that  $\Lambda= \Delta+\sigma$. It is easy to
verify that, if $x\in M_{\R}$, then $\Lambda_{x}=
\Delta_{x}+\sigma_{x}$.
Therefore, $\rec(\Lambda_{x})=\sigma_{x}$, which
implies the result. 
\end{proof}

\begin{lem} \label{lemm:16}
  Let $\Lambda $ be a polyhedron and $v\in \rec(\Lambda )$. Then
  there is a face $F$ of $\Lambda $ such that
  $v\in \ri (\rec(F))$. Moreover, 
  if $F_{1}$ and $F_{2}$ are two faces satisfying this
  condition, then $\rec(F_{1})=\rec(F_{2})$. 
\end{lem}
\begin{proof} This follows easily from Lemma \ref{lemm:12}.
\end{proof}

\begin{lem}[{\cite[Thm. 18.1]{Roc70}}]
\label{lemm:1}
  Let $C$ be a convex set and $F$ a face of $C$. If $D\subset C$ is
  a convex subset such that $\ri(D)\cap F\not =\emptyset$, then
  $D\subset F$.  
\end{lem}

\begin{proof}[Proof of Theorem \ref{thm:14}]
  We first prove \eqref{item:3}.
  The sets $\rec
  (\Lambda )$ for $\Lambda \in \Pi $ are polyhedra and, by Lemma
  \ref{lemm:12}, any face of an element of $\rec(\Pi )$ is in $\rec(\Pi
  )$. Thus, $\rec(\Pi )$ satisfies the first condition in Definition
  \ref{def:1} and 
 we have to show that it also satisfies the second one. That is, we
 have to show that
 given $\Lambda
  _{1},\Lambda _{2}\in \Pi $ such that $\rec(\Lambda _{1})\cap
  \rec(\Lambda _{2})$ is non-empty, this intersection is a common face of
  $\rec(\Lambda _{1})$ and $ 
  \rec(\Lambda _{2})$. 

  Let $v\in 
  \ri(\rec(\Lambda _{1})\cap \rec(\Lambda _{2}))$. Let $F_i$, $i=1,2$,
  be faces of $\Lambda _{i}$ satisfying the condition of Lemma
  \ref{lemm:16}. 
  We claim that
  we can find a 
  finite sequence of polyhedra $\Gamma _{1},\dots ,\Gamma _{n+1}$ with
  the following properties:
  \begin{enumerate}
  \item $F_{1}$ is a face of $\Gamma _{1}$ and $F_{2}$ is a face of
    $\Gamma _{n+1}$.
  \item The vector $v$ belongs to $\rec(\Gamma
    _{i}\cap \Gamma _{i+1})$ for $i=1,\dots,n$. 
  \end{enumerate}

  To prove this claim, we choose points $p_{i}\in \ri(F_{i})$, $i=1,2$. Since
  $|\Pi|$ is a connected polyhedral set, we can find a polygonal path
  $\gamma \colon [0,1] \to |\Pi |$ joining $p_{1}$ and $p_{2}$. By
  construction, $v\in \rec_{p_{1}}(|\Pi|)$. Since $|\Pi|$ satisfies the
  Minkowski-Weyl condition, $v\in \rec(|\Pi|)$. Therefore the map
  \begin{math}
    S\colon [0,1]\times \R_{\ge 0}\to |\Pi|,
    \end{math}
    given by
    \begin{math}
    S(t,r)=\gamma (t)+rv.
  \end{math}
  is well-defined. This is the key step in this proof and it fails if
  the hypothesis are not fulfilled: if $|\Pi |$ is not connected, then
  the polygonal path $\gamma $ may not exist and if $|\Pi |$ is
  connected but does not satisfy the Minkowski-Weyl condition, then the
  map~$S$ may not exist.

  Since $S$ is a piecewise affine function, we can find a finite covering
  $\mathfrak{U}$ of $[0,1]\times \R_{\ge 0}$ by  polyhedra, such
  that, for each $K\in \mathfrak{U}$, there is a $\Lambda \in \Pi$
  with $S(K)\subset \Lambda $. By the finiteness of
  $\mathfrak{U}$, we can find a number $l\ge 0$ such that the
  restriction of the covering $\mathfrak{U}$ to $[0,1]\times
  [l,\infty)$ consists of sets of the form $I_{\alpha }\times
  [l,\infty)$, where the $I_{\alpha }$ are closed intervals that cover $[0,1]$. We choose $I_{1},\dots, I_{n+1}$
  among them such that $0\in I_{0}$, $1\in I_{n+1}$ and $I_{i}\cap
  I_{i+1}\not =0$ for $i=1,\dots ,n$. For each $i$, we choose a polyhedron $\Gamma
  _{i}\in 
  \Pi$ such that $S(I_{i}\times [l,\infty))$ is contained in
  $\Gamma _{i}$.

  Since $\Gamma _{0}$ contains the point $p_{1}+lv$ and this point
  belongs to $\ri(F_1)$, then $F_{1}$ is a face of $\Gamma
  _{0}$. Analogously, $F_{2}$ is a face of $\Gamma _{n+1}$. By
  construction, it is also clear that $v\in \rec(\Gamma_{i}\cap \Gamma
  _{i+1}) $. Thus the claim is proved.

  For each $i=1,\dots ,n$ we choose a face $G_{i}$ of $\Gamma _{i}\cap \Gamma
  _{i+1}$ that satisfies the condition of Lemma
  \ref{lemm:16}. Applying Lemma \ref{lemm:16} to the polyhedra $\Gamma
  _{i}$ we obtain
  \begin{displaymath}
    \rec(F_{1})=\rec(G_{1})=\dots=\rec(G_{n})=\rec(F_{2}).
  \end{displaymath}
  By Lemma \ref{lemm:1},
  \begin{math}
    \rec(\Lambda _{1})\cap\rec(\Lambda _{2})\subset \rec(F_{1}).
  \end{math}
  Thus, we have the chain of inclusions
  \begin{displaymath}
    \rec(\Lambda _{1})\cap\rec(\Lambda _{2})\subset \rec(F_{1})\cap
    \rec(F_{2}) \subset\rec(\Lambda _{1})\cap\rec(\Lambda _{2}).
  \end{displaymath}
  Hence $\rec(\Lambda _{1})\cap\rec(\Lambda _{2})=\rec(F_{1})=\rec(F_{2})$ is a
  common face of $\rec(\Lambda _{1})$ and $\rec(\Lambda _{2})$. We
  conclude that
  $\rec(\Pi )$ is a complex.

  We now prove that $|\rec(\Pi )|=\rec(|\Pi |)$. On the one hand, for
  any $p\in |\Pi|$, we
  always have the chain of inclusions
  \begin{displaymath}
    \rec(|\Pi|)\subset \rec_{p}(|\Pi|)\subset \bigcup_{\Lambda
      \in\Pi}\rec(\Lambda )=|\rec(\Pi )|.  
  \end{displaymath}
  On the other hand, if $\Lambda \in \Pi$ and $p\in \Lambda $ we have
  \begin{math}
    \rec(\Lambda )\subset \rec_{p}(|\Pi|)= \rec(|\Pi|),
  \end{math}
  where the second equality follows from the Weyl-Minkowski
  condition. Thus
  \begin{displaymath}
    |\rec(\Pi )|=\bigcup_{\Lambda
      \in\Pi}\rec(\Lambda )\subset \rec(|\Pi|). 
  \end{displaymath}
  Hence the equality. 

  We next prove \eqref{item:4}. For a polyhedron $\Lambda $, we denote
  \begin{math}
      \cco(\Lambda ) =\R_{>0}(\Lambda \times \{1\}).
  \end{math}
  Observe that
  \begin{equation}
    \label{eq:69}
    \cc(\Lambda) =\cco(\Lambda
  )\sqcup (\rec(\Lambda )\times \{0\}).
  \end{equation}

  By Lemma \ref{lemm:12}, every face of $\rec(\Lambda )\times \{0\}$
  is of the form
  $\rec(F)\times \{0\}$ for some face $F$ of $\Lambda
  $. Moreover, a face of $\cc(\Lambda )$ is either 
  of the form $\cc(F)$ for a face $F$  of $\Lambda $ or a face of $\rec(\Lambda
  )\times \{0\}$. Hence  $\cc(\Pi)$ satisfies the first 
condition in Definition  \ref{def:1}. 

It remains to prove that any two elements of $\cc(\Pi)$ which are not
disjoint, intersect in a common face.
By \eqref{item:3}, the intersection of two cones of $\cc(
  \Pi )$ contained in $N_{\R}\times \{0\}$ is a common
  face. By \eqref{eq:69}, the same is true if we intersect a
  cone contained in $N_{\R}\times \{0\}$ with a cone of the form
  $\cc(\Lambda )$. If $\Lambda _{1},\Lambda _{2}\in \Pi $, 
  one verifies using \eqref{eq:69} that
  \begin{displaymath}
    \cc(\Lambda _{1})\cap \cc(\Lambda _{2})=
    \begin{cases}
      \cc(\Lambda _{1}\cap \Lambda _{2}), &\text{ if }\Lambda _{1}\cap
      \Lambda _{2}\not=\emptyset,\\
      (\rec(\Lambda _{1})\cap\rec(\Lambda _{2}))\times\{0\}, &\text{ otherwise.}
    \end{cases}
  \end{displaymath}
  In both cases, this is a common face of $\cc(\Lambda _{1})$ and
  $\cc(\Lambda _{2})$. Hence $\cc (\Pi )$ is a complex.

  Statement \eqref{item:5} follows easily from the previous ones.
\end{proof}  

\begin{exmpl} \label{exm:1}
  \begin{enumerate}
  \item Let $\Pi$ be the complex in $\R^{3}$ consisting in
    the set of faces of the polyhedra
  $
  \{(x_{1},x_{2},0)| x_{1},x_{2}\ge0\}$,   
  $\{(x_{1},x_{2},1)| x_{1}\ge x_{2}\ge0\}$  and  
  $\{(x_{1},x_{2},1)| x_{2}\ge x_{1}\ge0\}$.
This polyhedral
  complex satisfies the Minkowski-Weyl condition but
$|\Pi|$ is  not connected and $\rec(\Pi)$ is not a complex. Therefore, the 
  connectedness assumption  is
  necessary for the conclusion of Theorem \ref{thm:14}. 
\item \label{item:7} Let $\Pi$ be the complex in $\R^{3}$ consisting
  in the set of faces of the polyhedra
  $\{(x_{1},x_{2},0)| x_{1}\ge x_{2}\ge0\}$ and 
  $\{(x_{1},x_{2},1)| x_{2}\ge x_{1}\ge0\}$.
  This polyhedral
  complex does not satisfy the Minkowski-Weyl condition, $\rec(\Pi)$
  is a fan but
  $\rec(|\Pi|)\subsetneqq |\rec(\Pi)|$.   
\end{enumerate}
\end{exmpl}

Any  polyhedron satisfies the Minkowski-Weyl condition, therefore we have:
\begin{cor}\label{cor:8}
  Let $\Pi $ be a polyhedral complex in $N_{\R}$ such that $|\Pi
  |$ is convex. Then  $\rec(\Pi )$
  and $\cc(\Pi)$ are
  conic polyhedral complexes. If, in addition, $\Pi$ is rational
  (respectively 
  strongly convex) then $\rec(\Pi )$
  and $\cc(\Pi)$ are rational (respectively fans).
\end{cor}

Let $\Sigma$ be a conic polyhedral complex in
$N_{\R}\times\R_{\ge 0}$. We denote by  $\aff(\Sigma)$ the complex 
in $N_{\R}$ obtained by intersecting 
$\Sigma$ with the hyperplane $N_{\R}\times\{1\}$.  
Again, if $\Sigma  $ is rational or strongly convex, the same is true for
$\aff(\Sigma )$. 

\begin{cor} \label{cor:11} The correspondence $\Pi \mapsto \cc(\Pi )$
  is a bijection between the set of complete polyhedral complexes
  in $N_{\R}$ and the set of 
  complete conic polyhedral complexes in 
  $N_{\R}\times 
  \R_{\ge 0}$. Its
  inverse is the correspondence $\aff$. These bijections
  preserve rationality and strong convexity.
\end{cor}

\begin{proof} By Theorem \ref{thm:14}, if $\Pi $ is a complete
  polyhedral complex in $N_{\R}$ then $\cc(\Pi)$ is also a complete  conic
  polyhedral complex in $N_{\R}\times \R_{\ge0}$. Conversely, if $
  \Sigma $ is a complete conic polyhedral complex in
  $N_{\R}\times \R_{\ge 0}$,  it
  is clear that $\aff(\Sigma )$ is a complete
  polyhedral complex in  $N_{\R}$. 

  If $\Pi$ is a complete polyhedral complex in $N_{\R}$, then
  $\Pi=\aff(\cc(\Pi ))$. It remains to verify  that the other composition is
   the identity. Since we already know that $\cc(\aff(\Sigma ))$
  is a complex in $N_{\R}\times \R_{\ge 0}$, it is enough
  to show that $\Sigma $ and $\cc(\aff(\Sigma ))$ have the same cones
  of dimension $n+1$. But this is obvious.
Hence $\Sigma=\cc(\aff(\Sigma ))$. 

The last statement is clear.
\end{proof}

\begin{rem}
  \label{rem:9}
  \begin{enumerate}
  \item Example \ref{exm:17} shows that the map $\Pi\mapsto \cc(\Pi)$ does
    not produce a conic polyhedral complex from an arbitrary polyhedral
    complex. Thus the above corollary cannot be extended to
    arbitrary complexes.
  \item \label{item:6} Given a conic polyhedral
    complex $\Sigma $, the 
    polyhedral complex $\aff(\Sigma) $ does not need to satisfy
    the hypothesis of Theorem~\ref{thm:14}. Nevertheless, $\cc(\aff(\Sigma
    ))$ is the conic complex $\Sigma $. Therefore the hypothesis of
    Theorem~\ref{thm:14} are
    sufficient for the recession being a complex but they are not necessary.
    It would be interesting to 
    have a full characterization of the  complexes that arise as
    the image of $\aff$.
  \item If we restrict to polyhedral complexes with
    convex support, the correspondence 
    $\aff$ is not injective and 
     $\cc$ is only a right inverse of $\aff$. This is
    why in Corollary \ref{cor:11}, we restrict ourselves to complete
    complexes. 
  \end{enumerate}
\end{rem}

Let $\Pi $ be a polyhedral complex.
As we have seen in Remark \ref{rem:9}~\eqref{item:6}, the hypothesis of
Theorem \ref{thm:14} are not 
necessary for $\rec(\Pi )$ being a complex. A class of examples for
which this is true is that of extendable
complexes.

\begin{defn} \label{def:3} A polyhedral complex $\Pi $ in $\R^{n}$ is called
  \emph{extendable} if there exists a complete polyhedral complex
  $\overline \Pi$ such that $\Pi$ is a subcomplex of $\overline \Pi $.
\end{defn}

\begin{cor} \label{cor:1} 
  Let $\Pi $ be an extendable polyhedral complex. Then both $\rec(\Pi
  )$ and $\cc(\Pi )$ are complexes. 
\end{cor}
\begin{proof}
  If $\sigma \in \rec(\Pi )$, then it is clear that all the faces of
  $\sigma $ also belong to $\rec(\Pi )$.  Let $\overline \Pi$ be a
  complete polyhedral complex which contains $\Pi $.  Let $\sigma
  ,\tau \in \rec(\Pi )$. Since both belong to $\rec(\overline \Pi )$,
  by Theorem \ref{thm:14}, its intersection is a common face. Thus
  $\rec(\Pi )$ is a complex and the same is true for $\cc(\Pi )$.
\end{proof}

In particular, the complex from Example \ref{exm:17} can not be
extended to a complete polyhedral complex. By contrast, the complex
from Example \ref{exm:1}~\eqref{item:7} is extendable. This last
example shows that an extendable complex $\Pi $ does not necessarily
verify that $|\rec(\Pi )|=\rec(|\Pi |)$.

If one is willing to admit subdivisions, the issues raised by Example
\ref{exm:17} disappear. 

\begin{prop} \label{prop:2} Let $\Pi $ be a polyhedral complex in
  $\R^{n}$. Then there exists a subdivision $\Pi'$ of $\Pi $ that is
  extendable. In particular, $\rec(\Pi ')$ and $\cc(\Pi ')$ are
  complexes.
\end{prop}
\begin{proof}
  Let $\Lambda _{1},\dots,\Lambda _{m}$ be the polyhedra of $\Pi $. 
  For each $i$, the complex defined by the
  faces of $\Lambda_{i} $ is extendable. Denote by $\Pi _i $ any
  such extension. Then
  \begin{displaymath}
    \overline \Pi :=
    \{\Gamma  _{1}\cap \dots \cap \Gamma _{m}\mid \Gamma  _{i}\in \Pi
    _{i}\} 
  \end{displaymath}
  is a complete polyhedral complex that is a common subdivision of the
  $\Pi _{i}$. Let $\Pi '$ be the set of polyhedra of $\ov \Pi $ that
  are contained in $|\Pi |$. Then $\Pi '$ is an extendable subdivision
  of $\Pi $. The last statement follows from Corollary \ref{cor:1}.
\end{proof}

\section{Toric schemes over a DVR and tropical varieties}
\label{sec:toric-varieties-over}

Let $K$ be a field provided with a non-trivial discrete valuation
$\val\colon K^{\times}\twoheadrightarrow \Z$. 
Let  $K^{\circ}$ be its valuation ring and 
$S=\Spec(K^{\circ})$ its base scheme.
Let $\T_{S}\simeq \G_{m,S}^{n}$ be a
split torus over $S$ and  let $\T=\T_{S}\times \Spec(K)$ be the
corresponding split torus 
over $K$. Let $N=\Hom(\G_{m,K},\T)$ be 
the
lattice of one-parameter subgroups of $\T$. 

\begin{defn}\label{def:15}
A \emph{toric scheme over $S$ of relative dimension $n$} is a
normal integral
separated $S$-scheme of finite type $\cX$ equipped with an open
embedding $\T \hookrightarrow \cX\times \Spec(K)$ and an $S$-action of
$\T_{S}$ over $\cX$ that extends the
action of $\T$ on itself.
\end{defn}

Toric schemes over a DVR were introduced and
studied in
\cite{Kempfals:te}. 
In \emph{loc. cit.}, to each rational fan $\Sigma $ in $N_{\R}\times
\R_{\ge 0}$ 
it is associated a
toric scheme $\cX_{\Sigma }$ over $S$. Moreover, it is proved the
following classification theorem.

\begin{thm}[{\cite[\S~IV.3]{Kempfals:te}}]\label{thm:2}
  The correspondence $\Sigma \mapsto \cX_{\Sigma }$ 
is a bijection between the set of rational fans in $N_{\R}\times
\R_{\ge 0}$ and the set of 
    toric schemes over $S$ of relative dimension $n$. The scheme
    $\cX_{\Sigma }$ is proper if and only if $\Sigma$ is complete.
 \end{thm}

We are interested in the question of when an SCR polyhedral
complex in  $N_{\R}$ defines a toric scheme over $S$ and whether
this assignment allows to classify toric schemes over $S$.
Example \ref{exm:17} shows that an SCR polyhedral 
complex of $N_{\R}$ does not necessarily define a toric scheme
over $S$. Theorem \ref{thm:14} shows that to each SCR polyhedral 
complex $\Pi$ in $N_{\R}$ such that $|\Pi |$ is connected and satisfies the
Minkowski-Weyl condition we can associate a toric scheme $\cX_{\cc(\Pi )}$. Remark \ref{rem:9} shows that this assignment 
cannot give a classification in full generality. A direct
consequence of Theorem \ref{thm:2} and Corollary \ref{cor:11} is
 the following classification result. 

\begin{thm}\label{thm:7}
  The correspondence $\Pi\mapsto \cX_{\cc(\Pi)}$
  is a bijection between the set of complete SCR
  polyhedral complexes of $N_{\R}$ and the set of
  proper toric schemes over $S$ of relative dimension $n$.
 \end{thm}

Polyhedral complexes play also an important role in  tropical geometry.
Observe that the valuation of $K$ induces a map $ \T(\ov K)\to
N_{\R}$, which we also denote by $\val$.

\begin{defn}
  \label{def:2}
Let $V\subset \T$ be an algebraic set. The \emph{tropical variety}
associated to $V$, denoted $\trop(V)$, is the closure in $N_{\R}$ of
the subset  $\val(V(\ov K))$.
\end{defn}

Tropical varieties are polyhedral sets that  can be equipped with a standard  
structure of a polyhedral complex. Let $G_{V}$ denote the so-called
\emph{Gr\"obner complex} of $V$
\cite[\S~2]{MaclaganSturmfels:TropicalGeometry}, \cite[Defn. 5.5]{MR2503041}.
This is a complete rational polyhedral complex in $N_{\R}$. If the
stabilizer of $V$ is trivial, then $G_{V}$ is strongly convex.
The tropical variety $\trop(V)$ is the union of a finite number of
elements of $G_{V}$ and so it inherits a structure of a polyhedral
complex. We denote by $\Pi_{V}$ this extendable polyhedral complex. As a
consequence of Theorem \ref{thm:14} and Corollary \ref{cor:1} we deduce:

\begin{prop}
  \label{prop:1}
Let $V\subset \T$ be an algebraic set, then $\cc(\Pi_{V})$ is a
rational conic polyhedral complex. 
Moreover, 
let $G_{V}'$ be an SCR
subdivision of $G_{V}$ and let $\Pi'_{V}$ be the subdivision of
$\Pi_{V}$ induced by $G'_{V}$. 
Then $\cc(\Pi'_{V})$ is a rational fan. 
In particular, if the stabilizer of $V$ is trivial, then $\cc(\Pi_{V})$ is a
rational fan.
\end{prop}

Hence, if the stabilizer of $V$ is trivial, we can associate to $V$ the
toric scheme $\cX_{\cc(\Pi_{V})}$. This kind of  schemes are
interesting 
 because, by a theorem of J. Tevelev 
 extended by D. Speyer, 
the closure $\cV$ of $V$ in  $\cX$ is proper
over $S$ \cite[Prop.~2.4.1]{SpeyerPhD},
\cite[Prop.~3.2]{HelmKatz}. The scheme $\cV$ is called a tropical
compactification of $V$.

\bibliographystyle{amsalpha}

\newcommand{\noopsort}[1]{} \newcommand{\printfirst}[2]{#1}
\newcommand{\singleletter}[1]{#1} \newcommand{\switchargs}[2]{#2#1}
\providecommand{\bysame}{\leavevmode\hbox to3em{\hrulefill}\thinspace}
\providecommand{\MR}{\relax\ifhmode\unskip\space\fi MR } 
\providecommand{\MRhref}[2]{%
\href{http://www.ams.org/mathscinet-getitem?mr=#1}{#2} }
\providecommand{\href}[2]{#2}

\end{document}